\documentclass[12pt]{article}
\usepackage{amsmath, amsthm, amscd, amsfonts, amssymb, graphicx, color}
\usepackage[bookmarksnumbered, colorlinks, plainpages]{hyperref}
\usepackage[utf8]{inputenc}
\newtheorem{thm}{Theorem}[section]
\newtheorem{cor}[thm]{Corollary}
\newtheorem{lem}[thm]{Lemma}
\newtheorem{prop}[thm]{Proposition}
\newtheorem{defn}[thm]{Definition}

\numberwithin{equation}{section}

\newcommand{\be}{\begin{equation}}
\newcommand{\ee}{\end{equation}}
\newcommand{\ben}{\begin{enumerate}}
\newcommand{\een}{\end{enumerate}}
\newcommand{\beq}{\begin{eqnarray}}
\newcommand{\eeq}{\end{eqnarray}}
\newcommand{\beqn}{\begin{eqnarray*}}
\newcommand{\eeqn}{\end{eqnarray*}}

\newcommand{\bpf}{\begin{proof}}
\newcommand{\epf}{\end{proof}}
\newcommand{\bl}{\begin{lem}}
\newcommand{\el}{\end{lem}}
\newcommand{\bp}{\begin{prop}}
\newcommand{\ep}{\end{prop}}
\newcommand{\bd}{\begin{defn}}
\newcommand{\ed}{\end{defn}}
\newcommand{\bt}{\begin{thm}}
\newcommand{\et}{\end{thm}}

\begin{document}
\title{On compact Ricci solitons in Finsler geometry}
\author{B. Bidabad\footnote{The corresponding author} and M.Yar Ahmadi}
\date{}
\maketitle
\begin{abstract}
 Ricci solitons on Finsler spaces, previously
developed by the present authors, are  a generalization of Einstein spaces, which can be considered as a  solution to
the Ricci flow on compact Finsler manifolds. In the present work it is shown that on a Finslerian space, a forward complete shrinking Ricci soliton is compact if and only if it is bounded. Moreover, it is proved that a compact shrinking Finslerian Ricci soliton has finite fundamental group and hence the first de Rham cohomology group vanishes.
\end{abstract}
{\small\emph{Mathematics Subject Classification}: Primary 53C60; Secondary 53C44, 35C08.}
\section{Introduction}
The Ricci flow in Riemannian geometry was introduced by R. S. Hamilton in 1982, cf. \cite{Ha}, and since then
has been extensively studied thanks to its applications in geometry, physics and different branches of real world problems.
Quasi-Einstein metrics or Ricci solitons are considered as a solution to the Ricci flow equation and are subject of great interest in geometry and physics specially in relation with string theory, cf. \cite{FG}.

 Let $(M,g)$ be  a Riemannian manifold, a triple $(M,g,X)$ is said to be a \emph{quasi-Einstein metric} or \emph{Ricci soliton} if $g$ satisfies  the  equation
 \begin{align}\label{Eq;RST}
{2Ric+\mathcal{L}_{X} g=2\lambda g,}
\end{align}
where $Ric$ is the Ricci tensor, $X$  a smooth vector field on $M$, $\mathcal{L}_{X}$  the Lie derivative along $X$ and $\lambda$  a real constant. A Ricci soliton is said to be \emph{shrinking, steady} or \emph{expanding} if $\lambda>0$, $\lambda=0$ or $\lambda<0$, respectively. If the vector field $X$ is the gradient of a function $f$, then $(M,g,X)$ is said to be \emph{ gradient} and  (\ref{Eq;RST}) takes the familiar form
\begin{equation*}
{Ric+\nabla\nabla f=\lambda g.}
\end{equation*}
 On a compact Riemannian manifold a quasi-Einstein metric is a special solution to the Ricci flow equation defined by
\begin{align*}
\frac{\partial}{\partial t} g(t)=-2Ric, \quad {g(t=0):=g_{_0}}.
\end{align*}
J. Lott has shown that the fundamental group of a closed Riemannian manifold is finite for any gradient shrinking Ricci soliton, cf. \cite{Lott}. M.F. L\'{o}pez and E.G. R\'{i}o have proved that a Riemannian compact shrinking Ricci soliton has finite fundamental group, cf. \cite{FG}. W. Wylie has shown that a Riemannian complete shrinking Ricci soliton has finite fundamental group, cf. \cite{Wylie}.

The concept of Ricci flow  on Finsler manifolds is defined first by D. Bao, cf. \cite{bao}, using the Ricci tensor defined by H. Akbar-Zadeh,  \cite{Ak3}.
Recently the present authors have developed the concept of Ricci soliton on Finsler spaces as a generalization of Einstein spaces, cf. \cite{BY}.  It is proved that if there is a Ricci soliton on a compact Finsler manifold then there exists a solution to the Ricci flow equation and vice-versa. Since Finslerian Ricci solitons generalize Einstein manifolds, it is natural to ask whether classical results like Bonnet-Myers theorem for Finsler-Einstein manifolds of positive Ricci scalar remain valid for Finslerian Ricci solitons. In the present work,  in analogy with Riemannian space the shrinking Finslerian Ricci soliton is defined and it is shown that a forward complete shrinking Finslerian Ricci soliton $ (M,F,V) $ is compact if and only if $ \Vert V\Vert $ is bounded. Moreover, it is proved that in this case the fundamental group is finite and as a consequence the first de Rham cohomology group of $M$ vanishes.

\section{Preliminaries and notations}
 Let $M$ be a real n-dimensional differentiable  manifold. We denote by $TM$ its tangent bundle and by $\pi :TM_{0} \longrightarrow M$,  the fiber bundle of non zero tangent vectors.
A \emph{Finsler structure} on $M$ is a function $F:TM\longrightarrow [0,\infty)$, with the following properties:\\
I. Regularity: $F$ is $C^{\infty}$ on the entire slit tangent bundle $TM_{0}=TM\backslash 0$.\\
II. Positive homogeneity: $F(x,\lambda y)=\lambda F(x,y)$ for all $\lambda > 0$.\\
III. Strong convexity: The $n\times n$ Hessian matrix $g_{ij}=([\frac{1}{2} F^{2}]_{y^{i} y^{j}})$ is positive definite at every point of $TM_{0}$.
A \emph{Finsler manifold} $(M,F)$ is a pair consisting of a differentiable manifold $M$ and a Finsler structure $F$.
The formal Christoffel symbols of  second kind and spray coefficients are denoted  respectively by
\begin{align*}
\gamma^i_{jk}:=g^{is}\frac{1}{2}\big(\frac{\partial g_{sj}}{\partial x^k}-\frac{\partial g_{jk}}{\partial x^s}+\frac{\partial g_{ks}}{\partial x^j}\big),
\end{align*}
 where $g_{ij}(x,y)=[\frac{1}{2} F^2]_{y^{i}y^{j}}$, and
 $
 G^i:=\frac{1}{2}\gamma^i_{jk}y^j y^k.
$
We consider also the \emph{reduced curvature tensor} $R^i_k$ which is expressed entirely in terms of the $x$ and $y$ derivatives of spray coefficients $G^i$.
 \begin{align}\label{E,Ricci scalar}
 R^{i}_{k}:=\frac{1}{F^2}\big(2\frac{\partial G^i}{\partial x^k}-\frac{\partial^2 G^i}{\partial x^j \partial y^k}y^j +2G^j\frac{\partial^2 G^i}{\partial y^j \partial y^k} - \frac{\partial G^i}{\partial y^j}\frac{\partial G^j}{\partial y^k}\big).
\end{align}
In the general Finslerian setting, one of the Ricci tensors is  introduced by H. Akbar-Zadeh \cite{AZ} as follows.
\begin{equation}\label{E;Ricci flow}
{Ric_{jk} :=[\frac{1}{2} F^{2} \mathcal{R}ic]_{y^{j} y^{k}}},
\end{equation}
where
$
\mathcal{R}ic=R^{i}_{i}
$
and  $R^i_k$ is defined by (\ref{E,Ricci scalar}). Akbar-Zadeh's definition of Einstein-Finsler space related to this Ricci tensor is obtained as a critical point of an Einstein-Hilbert functional and, (see \cite{Ak3} chapter IV). One of the advantages of the Ricci  quantity defined here  is its independence of the choice of Cartan, Berwald or Chern (Rund) connections. Based on the Akbar-Zadeh's Ricci tensor, in analogy with the equation (\ref{E;Ricci flow})
, D. Bao has considered, the following natural extension of \emph{Ricci flow } in Finsler geometry, cf.  \cite{bao},
\begin{align*}
\frac{\partial}{\partial t} g_{jk}=-2Ric_{jk}, \quad {g(t=0):=g_{_0}}.
\end{align*}
This equation is equivalent to the following differential equation
\begin{align*}
{\frac{\partial}{\partial t}(\log F(t))=-\mathcal{R}ic,}\quad {F(t=0):=F_{_0}},
\end{align*}
where, $F_{_0}$ is the initial Finsler structure.
Let $V=V^{i}(x) \frac{\partial}{\partial x_i}$  be a vector field on $M$.

The Lie derivative of a Finsler metric tensor $ g_{jk} $ is given in the following tensorial form by
\begin{align}\label{Eq;Lieder}
\mathcal{L}_{\hat V}g_{jk}=\nabla_jV_k+\nabla_kV_j+2(\nabla_0V^l)C_{ljk},
\end{align}
where $\hat V$ is the complete lift of a vector field $V$ on $M$, $ \nabla $ is the Cartan connection, $ \nabla_0=y^p\nabla_p $ and $ \nabla_p=\nabla_{\frac{\delta}{\delta x^p}} $, (see \cite{Yano} at p. 180 and see \cite{BJ}).\\

Let $M$ be a connected smooth manifold, then there exists a simply connected smooth manifold $\tilde M$, called the universal covering manifold of $M$, and a smooth covering map $p:\tilde M\longrightarrow M$ such that it is unique up to a diffeomorphism. The complete lift of $p$ is a map $\bar p:T\tilde M\longrightarrow TM$ is given by $$ \bar p(\tilde x,\tilde y)=(p(\tilde x),\tilde y^i\frac{\partial p}{\partial \tilde x^i})=(p(\tilde x),\tilde y^i\frac{\partial p^j}{\partial \tilde x^i}\frac{\partial}{\partial x^j}), $$
where $\tilde y\in T_{\tilde x}\tilde M$.
\section{Shrinking Finslerian Ricci soliton}
Let $(M,F_0)$ be a Finsler manifold and $V=V^{i}(x) \frac{\partial}{\partial x^{i}}$
 a vector field on $M$. We call the triple $(M,F_0,V)$ a Finslerian \emph{quasi-Einstein} or a \emph{Finslerian Ricci soliton} if ${g_{jk}}$ the Hessian related to the Finsler structure $F_0$ satisfies
 \begin{align}\label{Eq;DefRicciSoliton+}
2Ric_{jk}+\mathcal{L}_{\hat{V}}  {g_{jk}}=2 \lambda g_{jk},
\end{align}
where,   $\hat V$ is the complete lift of $V$ and  $\lambda \in {\mathbb{R}}$.  A Finslerian Ricci soliton is said to be \emph{shrinking, steady} or \emph{expanding} if $\lambda>0$, $\lambda=0$ or $\lambda<0$, respectively. The Finslerian Ricci soliton is said to be forward complete (resp. compact) if $ (M,F_0) $ is forward complete (resp. compact). Note that according to the Hopf-Rinow's theorem,  the both notions forward complete and forward geodesically complete are equivalent. Denote by $SM$ the sphere bundle, defined by $SM:=\bigcup\limits _{x\in M} S_xM$ where $S_xM:=\{y\in T_xM | F(x,y)=1\}$. For a vector field $ X=X^i{(x)}\frac{\partial}{\partial x^i} $ on $ M $ define
\begin{align}
\Vert X\Vert_x=\max\limits_{y\in S_xM}\sqrt{g_{ij}(x,y)X^iX^j},
\end{align}
where $x\in M$ (see \cite{bcs} at p. 321). Since $ S_xM $ is compact, $ \Vert X\Vert_x $ is well defined.
\begin{thm}\label{theorem1}
Let $(M,F_0)$ be a  forward geodesically complete Finsler manifold satisfying
\begin{align}\label{Eq;a}
2Ric_{jk}+\mathcal{L}_{\hat{V}}  {g_{jk}}\geq 2 \lambda g_{jk},
\end{align}
where, $\lambda > 0$.
  Then, $M$ is compact if and only if $\Vert V\Vert $ is bounded on $M$ and moreover, in such a case, $diam(M)\leq \frac{\pi}{\lambda}\big(D+\sqrt{D^2+\lambda(n-1)}\big)$.
\end{thm}
\begin{proof}
Let $ M $ be a compact manifold, it is clear that $\Vert V\Vert $ is bounded on $M$. Conversely, let $ p,q $ be two points in $ M $ joined by a minimal geodesic $\gamma$ parameterized by the arc length $t$, $\gamma:[0,\infty)\longrightarrow M$. Using (\ref{Eq;Lieder}) we have along $ \gamma $
\begin{align}\label{Eq;1}
{\gamma^{\prime}}^j{\gamma^{\prime}}^k\mathcal{L}_{\hat V}g_{jk}={\gamma^{\prime}}^j{\gamma^{\prime}}^k\big(\nabla_jV_k+\nabla_kV_j+2(\nabla_0V^l)C_{ljk}\big).
\end{align}
Along $ \gamma $, we have $ {\gamma^{\prime}}^j{\gamma^{\prime}}^k(\nabla_0V^l)C_{ljk}(\gamma(t),\gamma^{\prime}(t))=0 $. Hence (\ref{Eq;1}) reduces to
\begin{align}\label{Eq;2}
{\gamma^{\prime}}^j{\gamma^{\prime}}^k\mathcal{L}_{\hat V}g_{jk}=2{\gamma^{\prime}}^j{\gamma^{\prime}}^k\nabla_jV_k.
\end{align}
On the other hand, by compatibility of metric with the Cartan connection, we have along the geodesic $\gamma$
\begin{align}\label{Eq;3}
{\gamma^{\prime}}^j{\gamma^{\prime}}^k\nabla_jV_k=
\nabla_{{\gamma^{\prime}}^j\frac{\delta}{\delta x^j}}({\gamma^{\prime}}^kV_k)=
\nabla_{\hat\gamma^{\prime}}({\gamma^{\prime}}^kV_k)=
\frac{d}{dt}({\gamma^{\prime}}^kV_k),
\end{align}
where $\hat\gamma^{\prime}={\gamma^{\prime}}^j\frac{\delta}{\delta x^j}$. Replacing (\ref{Eq;3}) in (\ref{Eq;2}) we have
\begin{align}\label{Eq;4}
{\gamma^{\prime}}^j{\gamma^{\prime}}^k\mathcal{L}_{\hat V}g_{jk}=2\frac{d}{dt}({\gamma^{\prime}}^kV_k).
\end{align}
By means of (\ref{Eq;a}) and (\ref{Eq;4}) we get
\begin{align*}
2{\gamma^{\prime}}^j{\gamma^{\prime}}^kRic_{jk}+2\frac{d}{dt}({\gamma^{\prime}}^kV_k)\geq 2\lambda {\gamma^{\prime}}^j{\gamma^{\prime}}^kg_{jk}.
\end{align*}
By last inequality we conclude that
\begin{align*}
{\gamma^{\prime}}^j{\gamma^{\prime}}^kRic_{jk}\geq\lambda {\gamma^{\prime}}^j{\gamma^{\prime}}^kg_{jk}-\frac{d}{dt}({\gamma^{\prime}}^kV_k)=\lambda+\frac{d}{dt}(-{\gamma^{\prime}}^kV_k).
\end{align*}
On the other hand, by means of the Cauchy-Schwarz inequality, we have along the geodesic $\gamma$
\begin{align*}
|-{\gamma^{\prime}}^kV_k|&=|{\gamma^{\prime}}^kV_k|=|g_{kl}(\gamma(t),\gamma^{\prime}(t)){\gamma^{\prime}}^kV^l|\\&\leq |g_{pq}(\gamma(t),\gamma^{\prime}(t)){\gamma^{\prime}}^p{\gamma^{\prime}}^q|^{\frac{1}{2}}|g_{rs}(\gamma(t),\gamma^{\prime}(t))V^rV^s|^{\frac{1}{2}}\\&\leq \max\limits_{y\in S_{_{_{\gamma(t)}}}M}|g_{rs}(\gamma(t),y)V^rV^s|^{\frac{1}{2}}=\Vert V\Vert_{_{\gamma(t)}}.
\end{align*}
Since $ \Vert V\Vert $ is assumed to be bounded on $ M $, there exists a positive constant $ D $ such that $ \Vert V\Vert_{_{\gamma(t)}}\leq D $ and therefore  along $ \gamma $, $ |-{\gamma^{\prime}}^kV_k|\leq D $. Now, the result follows from generalization of Myer's theorem, cf. \cite{Anas}. That is, $M$ is compact and moreover it is bounded from above by $diam(M)\leq \frac{\pi}{\lambda}(D+\sqrt{D^2+\lambda(n-1)})$. This completes the proof.
\end{proof}

\begin{cor}
Let $(M,F,V)$ be a forward complete shrinking Finslerian Ricci soliton. Then, $M$ is compact if and only if $ \Vert V\Vert $ is bounded on $M$ and moreover, in this case, $diam(M)\leq \frac{\pi}{\lambda}(D+\sqrt{D^2+\lambda(n-1)})$.
\end{cor}

\begin{thm}\label{thm2}
Let $(M,F)$ be a compact Finsler manifold satisfying (\ref{Eq;a}). Then the fundamental group $\pi_1(M)$ of $M$ is finite and its first cohomology group vanishes, i.e.,  $H^1_{dR}(M)={0}$.
\end{thm}
\begin{proof}
Let $ \tilde M $ be the universal covering manifold of $M$ with the smooth covering map $ p:\tilde M\longrightarrow M $.  Pull back of complete lift of the smooth covering map $p$, i.e., $\bar p^*F:=F\circ \bar p:T\tilde M\longrightarrow [0,\infty) $ is a Finsler structure on $\tilde M$. In fact, we  check simply, the three conditions of Finsler structure. We have the regularity condition since $F$ and $p$ are $C^{\infty}$, and so is $\bar p^*F $. Next,
\begin{align*}
\bar p^*F(x,\lambda y)&=F\circ \bar p(x,\lambda y)=F(p(x),\lambda y^i\frac{\partial p}{\partial x^i})\\&=\lambda F(p(x),y^i\frac{\partial p}{\partial x^i})=\lambda\bar p^*F(x,y).
\end{align*}
Thus the positive homogeneity is satisfied. Finally, assume that $\bar p^*x^i=\tilde x^i $ and $\bar p^*y^i=\tilde y^i $. For strong convexity we have
\begin{align*}
\tilde g_{ij}:=[\frac{1}{2}(\bar p^*F)^2]_{\tilde y^i\tilde y^j}=\frac{1}{2}\frac{\partial^2((\bar p^*F)^2)}{\partial \tilde y^i\partial \tilde y^j}=\frac{1}{2}\frac{\partial^2(\bar p^*F^2)}{\partial \tilde y^i\partial \tilde y^j}.
\end{align*}
One can easily check that
$$\frac{\partial(\bar p^*F^2)}{\partial \tilde y^i}=\bar p^*\frac{\partial F^2}{\partial y^i},$$
from which
\begin{align}\label{Eq:pd}
\tilde g_{ij}=[\frac{1}{2}(\bar p^*F)^2]_{\tilde y^i\tilde y^j}=\frac{1}{2}\frac{\partial^2(\bar p^*F^2)}{\partial \tilde y^i\partial \tilde y^j}=\bar p^*[\frac{1}{2}F^2]_{ y^i y^j}=\bar p^*g_{ij}.
\end{align}
Using the facts that $[\frac{1}{2}F^2]_{ y^i y^j}$ is positive definite on $TM_0$ and $ \bar p^* $ is a local diffeomorphism (note that $p$ is the smooth covering map), $\bar p^*[\frac{1}{2}F^2]_{ y^i y^j}$ is also positive definite on $T\tilde M_0$ and hence $\tilde F:=\bar p^*F $ defines a Finsler structure on $T\tilde M_0$. Moreover $(\tilde M,\tilde F)$ is locally isometric to $(M,F)$. Let $W$ denote the lift of $V$, that is, $W:=p^*V=(p^{-1})_*V$. More precisely , since $p$ is a local diffeomorphism, we can define $W:=p^*V=(p^{-1})_*V$. By means of the local isometry $p:(\tilde M,\tilde F)\longrightarrow (M,F)$ and the inequality (\ref{Eq;a}), we have
\begin{align*}
\bar p^*(2Ric_{jk}+\mathcal{L}_{\hat{V}}  {g_{jk}})\geq 2 \bar p^*(\lambda g_{jk}).
\end{align*}
By linearity of $\bar{p}^*$ we get
\begin{align}\label{ineqaulity}
2\bar p^*Ric_{jk}+\bar p^*\mathcal{L}_{\hat{V}}  {g_{jk}}\geq 2\lambda \bar p^*( g_{jk}).
\end{align}
By means of (\ref{Eq:pd}), $W=p^*V$ and commutativity of Lie derivative  and the pull back $\bar{p}^*$ we obtain
\begin{align}\label{Eq;5}
 \bar p^*\mathcal{L}_{\hat{V}}  {g_{jk}}=\mathcal{L}_{\hat{W}}  {\tilde g_{jk}}.
\end{align}
On the other hand, one can easily check that $\tilde Ric_{jk}=\bar p^*Ric_{jk}$. In fact we have
\begin{align*}
\bar p^*Ric_{jk}&=\bar p^*[\frac{1}{2} F^{2} \mathcal{R}ic]_{y^{j} y^{k}}=\frac{1}{2}\bar p^*\frac{\partial^2(F^{2} \mathcal{R}ic)}{\partial y^i\partial y^j}\\&=\frac{1}{2}\frac{\partial^2}{\partial\tilde y^i\partial \tilde y^j}(\bar p^*(F^{2} \mathcal{R}ic))=\frac{1}{2}\frac{\partial^2}{\partial\tilde y^i\partial\tilde y^j}\Big(\bar p^*(F^{2})\bar p^*(\mathcal{R}ic)\Big).
\end{align*}
Since $\bar p^*(\mathcal{R}ic)=\tilde{\mathcal{R}}ic$, cf. \cite{BY}, and $ \bar p^*(F^{2})=\tilde F^2 $, we get
\begin{align}\label{Eq;6}
\bar p^*Ric_{jk}=\frac{1}{2}\frac{\partial^2}{\partial\tilde y^i\partial\tilde y^j}\Big(\bar p^*(F^{2})\bar p^*(\mathcal{R}ic)\Big)=\frac{1}{2}\frac{\partial^2}{\partial\tilde y^i\partial\tilde y^j}(\tilde F^{2}\tilde{ \mathcal{R}}ic)=\tilde Ric_{jk}.
\end{align}
Replacing (\ref{Eq:pd}), (\ref{Eq;5}) and (\ref{Eq;6}) in (\ref{ineqaulity}), leads to
\begin{align*}
2\tilde Ric_{jk}+\mathcal{L}_{\hat{W}}  {\tilde g_{jk}}\geq 2\lambda \tilde  g_{jk}.
\end{align*}
On the other hand we have
\begin{align}\label{norm}
\Vert W\Vert_{\tilde x}=&\max\limits_{\tilde y\in S_{\tilde x}\tilde M}\big((\bar p^*g_{ij})(\tilde x,\tilde y)W^iW^j\big)^\frac{1}{2}=\max\limits_{\tilde y\in S_{\tilde x}\tilde M}\big(g_{ij}(p(\tilde x),\tilde y^i\frac{\partial p}{\partial \tilde x^i})\bar p_*W^i\bar p_*W^j\big)^\frac{1}{2}\nonumber\\&\leq \max\limits_{y\in S_{\tilde x}\tilde M}\big(g_{ij}(p(\tilde x),y)\bar p_*W^i\bar p_*W^j\big)^\frac{1}{2}=\Vert\bar p_*W \Vert_{_{p(\tilde x)}}.
\end{align}
 By compactness of $M$, the norm $\Vert \bar p_*W\Vert$ is bounded on $M$ and therefore by means of (\ref{norm}) the norm $\Vert W\Vert$ is bounded on $\tilde M$. It follows from Theorem \ref{theorem1} that $(\tilde M,\tilde F)$ is compact. Thus the closed subset $p^{-1}(x)$ of $\tilde M$ is compact and being discrete is finite. By assumption, $M$ is connected, so all of its fundamental groups $\pi_1(M,x)$ are isomorphic, where $x$ denotes the base point. Since $\tilde M$ is a universal cover, $\pi_1(M,x)$ is bijective with $p^{-1}(x)$ and therefore $\pi_1(M)$ is finite. Thus, by a well known result the first cohomology group $H^1_{dR}(M)=0$. This completes the proof.
\end{proof}
\begin{cor}
Let $(M,F,V)$ be a compact shrinking Finslerian Ricci soliton. Then the fundamental group $\pi_1(M)$ of $M$ is finite and therefore $H^1_{dR}(M)=0$.
\end{cor}
\begin{cor}
Let $(M,F,V)$ be a compact shrinking Finslerian Ricci soliton. Then the fundamental group $\pi_1(SM)$ of $SM$ is finite and therefore $H^1_{dR}(SM)=0$.
\end{cor}
\begin{proof}
Let $ \tilde M $ be the universal covering manifold of $M$ with the smooth covering map $p:\tilde M\longrightarrow M$. It is well known that the homotopic sequence of the fibre bundle $(S\tilde M,\tilde\pi,\tilde M,S^{n-1})$ is exact. That is
\begin{align}\label{sequence}
\cdots\longrightarrow\pi_1(S^{n-1})\longrightarrow\pi_1(S\tilde M)\longrightarrow\pi_1(\tilde M)\longrightarrow\cdots,
\end{align}
is exact. Since $\tilde M$ is simply connected, $\pi_1(\tilde M)=0$. We know that $\pi_1(S^{n-1})=0$. Thus by (\ref{sequence}) we get $\pi_1(S\tilde M)=0$. One can easily check that $\bar p:S\tilde M\longrightarrow SM $ is a smooth covering map. Therefore $ S\tilde M $ is the universal covering manifold of $SM$. According to the proof of Theorem \ref{thm2}, $\tilde M$ is compact and so is $S\tilde {M}$. Thus the fundamental group $\pi_1(SM)$ is finite and therefore $H^1_{dR}(SM)=0$.
\end{proof}



 Faculty of Mathematics, Amirkabir University of Technology  (Tehran Polytechnic), Tehran, Iran\\
 bidabad@aut.ac.ir\\
 m.yarahmadi@aut.ac.ir
\end{document}